\documentclass[11pt]{article}

\usepackage{amsmath,amsfonts,amssymb,amsthm}
\usepackage{graphicx}
\usepackage{subcaption}
\usepackage{epstopdf}
\usepackage{algorithmic}
\usepackage{enumerate}
\usepackage{amscd,latexsym,authblk,url}
\usepackage{multirow}
\usepackage{booktabs}
\usepackage{setspace}

\newtheorem{theorem}{Theorem}[section]
\newtheorem{proposition}[theorem]{Proposition}

\newtheorem{definition}[theorem]{Definition}
\newtheorem{lemma}[theorem]{Lemma}
\newtheorem{example}[theorem]{Example}
\newtheorem*{claim}{Claim}

\title{Metric representations by minimal graphs\thanks{
Partially supported by grants 
PID2021-123278OB-I00 and
PID2023-150725NB-I00, both funded by the Spanish Ministry of Science and Innovation MICIU/AEI/10.13039/501100011033, ``ERDF A way of making Europe''and Gen. Cat. DGR SGR2023 (2021SGR00266).}}

\author[1]{V\'ictor Franco-Sánchez}
\author[2]{Merc\`e Mora}
\author[3]{Mar\'ia Luz Puertas}

\affil[1]{Universitat Polit\`ecnica de Catalunya

\tt{victor.franco.sanchez@upc.edu}}

\affil[2]{Universitat Polit\`ecnica de Catalunya

  \tt{merce.mora@upc.edu}}

\affil[3]{Universidad de Almer\'ia 

\tt{mpuertas@ual.es}}

\begin{document}
\date{}
\maketitle

\begin{abstract}
A resolving set in a graph $G$ is a vertex subset $W= \{\omega^1, \dots, \omega^n\} \subseteq V(G)$ such that each $u \in V(G)$ can be uniquely identified by the vector $r(u \vert W) = (d(u,\omega^1), \dots, d(u,\omega^n))$ of metric coordinates of $u$ with respect to $W$.
The reverse problem of identifying the vector sets that are a set of coordinates of some graph provides the concept of realizable vector set $S \subset \mathbb{Z}^n$ by a pair $(G, W)$ meaning that $S=\{ r(u\vert W)\colon u\in V(G)\}$ with $W$ a resolving set of the graph $G$. Here we focus on the minimality of the realizations of vector sets with respect to their edge sets. On the one hand, we study conditions under which it is possible to remove an edge from the graph and keep the realizability condition. This provides a method for finding minimal realizations, as well as allowing us to characterize uniquely realizable vector sets. On the other hand, we prove that the decision problem of realizing a vector set by a graph with a given number of edges is an NP-complete problem. Finally, we characterize the vector sets that are realizable by a tree and, furthermore, we study the case in which such a realization is the only one.
\end{abstract}



\section{Introduction}\label{Sec:intro}

 The seminal paper of Harary and Melter~\cite{Harary1976} about metric dimension in graphs begins by recalling the definition of a basis of a metric space, taken from Blumenthal's classic book~\cite{Blumenthal1953}. Let $S=\{v_1, \dots, v_n\}$ be an ordered set of points in a metric space $M$. The $S$-coordinates of a point $u$ are given by the distance vector $[d(u,v_1), \dots d(u,v_n)]$. The set $S$ is called a metric basis for $M$ if (1) no two  points of $M$ have the same coordinates and (2) there is no smaller set satisfying (1). The authors then state that similar definitions apply to a connected graph, and the concepts of metric coordinates, metric basis and metric dimension in graphs arises. 

 Blumenthal's book, as stated at the beginning of Chapter II, deals with the geometric properties in which the distance between two points is invariant, and what the author calls \emph{distance geometry}. In other words, the focus is on the distances between points in a space and what properties can be deduced from them. This viewpoint has been extensively studied in metric spaces and also in graphs. 
 
 The problem of studying the properties of a connected graph from the distances between its vertices was first introduced in~\cite{Hakimi1965} and since then has been approached in several ways. For instance, by using the distance matrix, in~\cite{Smoleskii1962} the author considers the problem of defining a tree through a set of distances between its leaves and proved that if such a tree exists, then it is unique. Following theses ideas, in~\cite{Zaretskii65} the author establishes a necessary and sufficient condition for the existence of such trees and provides an algorithm for their construction. The same problem has been studied using other vertex sets, different from the set of leaves of a tree, but that can represent, in some sense, the ``extremities'' of a graph. Thus, in~\cite{Caceres2025} the authors investigate how to reconstruct a graph from the distances between its boundary vertices.

Our approach to this problem is inspired in the Blumenthal's original ideas and the so-called \emph{distance geometry problem}~\cite{Liberti2017}: ``is there a set of points in a vector space of given dimension which yields a given subset pairwise distances, each with its assigned point name pair?'' In this paper, we study properties of graphs that can be deduced from the metric coordinates of the vertices, a similar starting point to that of~\cite{Feit2018}. 

 Given a graph $G$, a {\it resolving set} is a subset of vertices $W \subseteq V(G)$ such that for any pair of distinct vertices $u, v \in V(G)$, there exists $\omega \in W$ that {\it resolves} the pair $u,v$, that is, $d(u,\omega) \neq d(v,\omega)$  (see \cite{Harary1976, Slater1975}). So, each vertex $u \in V(G)$ can be uniquely identified by its {\it metric coordinates} $r(u \vert W) = (d(u,\omega^1), \dots, d(u,\omega^n))$ with respect to the set $W = \{\omega^1, \dots, \omega^n\}$. The problem of identifying sets that are in fact sets of metric coordinates of some graph gives rise to the concept of realizability. A finite set $S \subset \mathbb{Z}^n$ is said to be {\it realizable} by a pair $(G, W)$, where $G$ is a graph and $W$ is a resolving set, if $S=\{ r(u\vert W)\colon u\in V(G)\}$. Realizable sets were characterized in~\cite{MoraPuertas2023} and the authors also studied the uniqueness of the realizations in the case $n=2$.

In this work, we continue this line of research. In Section~\ref{sec:minimal}, we explore under what conditions one realization can be derived from another one by adding or removing some edges, which leads us to consider maximal and minimal realizations and characterize the sets such that the realization is unique. 
Moreover, in Section~\ref{sec:complexity} we study the complexity of a decision problem related with the minimality of the metric realizations. 
Specifically, the decision problem of, given a set of distance vectors and a constant $k$,
determine if it is realizable by a graph with at most $k$ edges.
Finally, in Section~\ref{sec:trees}, we focus on realizations by trees, characterizing the coordinate sets that are realizable by trees, as well as the uniqueness of such realizations.

We include here some definitions and known results that will be useful in the rest of the paper. All graphs considered in this paper are finite and connected. The following are standard notions in graph theory. Given a graph $G$, with vertex set $V(G)$ and edge set $E(G)$, the neighborhood of a vertex $u\in V(G)$ is $N_G(u)=\{v\in V(G)\colon uv\in E(G)\}$. If $u,v$ are distinct vertices of $G$ such that $uv\notin E(G)$, we define $G+uv$ as the graph with vertex set $V(G+uv)=V(G)$ and edge set $E(G+uv)=E(G)\cup \{ uv\}$. Moreover, if $uv\in E(G)$, we define $G-uv$ as the graph with vertex set $V(G-uv)=V(G)$ and edge set $E(G-uv)=E(G)\setminus \{ uv\}$.
For arbitrary graphs $G$ and $H$, a map $f:V(H)\mapsto V(G)$ is an {\it isometric embedding of $H$ into $G$} if $d_G(f(u),f(v))=d_H(u,v)$, for any $u,v\in V(H)$.

\begin{definition}\cite{Chartrand2000,Harary1976,Slater1975}
Given a graph $G$, the metric representation of a vertex $u\in V(G)$ with respect to an ordered vertex subset $W=\{\omega^1, \dots , \omega^n\}\subseteq V(G)$ is $r(u\vert W)=(d(x,\omega^1), \dots d(x,\omega^n))$. A vertex subset $W$ of a graph $G$ is a resolving set if $r(u\vert W)\neq r(v\vert W)$, for each pair of vertices $u,v\in V(G)$, with $u\neq v$. 

\end{definition}

\begin{definition}\cite{MoraPuertas2023}
A subset $S\subset \mathbb{Z}^n$ is realizable if there exists a graph $G$  and a resolving set $W$ of $G$ such that $S=\{r(u\vert W)\colon u\in V(G)\}$.
In such a case, we say that $(G,W)$ is a \emph{realization} of $S$.
\end{definition}

\begin{theorem}\label{theo:realizable}
\cite{MoraPuertas2023}
A subset $S\subset \mathbb{Z}^n$ is realizable if and only if the following properties hold.
\begin{enumerate}
    \item For every $x\in S$ and every $i\in [n]$,  $x_i\geq 0$. Moreover, $x$ has at most one coordinate equal to zero.
    \item For every $i\in [n]$, there exists exactly one element $x\in S$ such that $x_i=0$.
    \item If $x\in S$ and $x_i > 0$ for some $i\in [n]$, then there exists $y\in S$ satisfying $y_i=x_i-1$ and $\max_{j\in [n]} \{\vert y_j-x_j \vert \}\leq 1$.
\end{enumerate}
\end{theorem}

\begin{definition}\label{def:equivalent}
\cite{MoraPuertas2023}
Two realizations $(G,W)$ and $(G',W')$ of a set $S\subset \mathbb{Z}^n$ are equivalent if the map $f\colon V(G)\to V(G')$ satisfying $r(u\vert W)=r(f(u)\vert W')$ is a graph isomorphism.
\end{definition}

\begin{definition}\label{def:canonical}
\cite{MoraPuertas2023}
Let $S\subset \mathbb{Z}^n$ be a realizable set. The \emph{canonical realization} of $S$ is $(\widehat{G},\widehat{W})$,
where
$$V(\widehat{G})=S, \ \ E(\widehat{G})=\{ xy\colon x,y\in S \textrm{ and } \max_{i\in [n]} \vert y_i-x_i\vert =1\},$$
$$\widehat{W}=\{x\in S\colon x_i=0, {\rm \ for\ some\ } i\in [n]\}.$$
\end{definition}

The canonical realization plays a central role among the realizations of $S$, since all of them can be included into the canonical one, in the following way.

\begin{proposition}\label{prop:embeding}\cite{MoraPuertas2023}
    Let $(G,W)$ be a realization of $S\subset \mathbb{Z}^n$. Then, the map $\phi\colon V(G)\to V(\widehat{G})$ defined by $\phi(u)=r(u\vert W)$ is a bijection. Moreover, the (not necessarily induced) subgraph $G^*$ of $\widehat{G}$ defined by $V(G^*)=V(\widehat{G})=S$ and $E(G^*)=\{\phi(u)\phi(v)\colon uv\in E(G)\}$, provides  a realization $(G^*,\widehat{W})$ of $S$ equivalent to $(G,W)$.
\end{proposition}

Given a realization $(G,W)$ of a set $S$, the map $\phi$ defined in Proposition~\ref{prop:embeding} is obviously an isomorphism between $G$ and a spanning subgraph of the canonical realization. 
Hence, from now on we will consider  all the realizations of a set $S$ as spanning subgraphs of $\widehat{G}$ by identifying the realization $(G,W)$ with the realization $(G^*,\widehat{W})$ by means of the map $\phi$. 

Notice that the map $\phi$ is not always an isometric embedding of $G$ into the canonical realization $\widehat{G}$, 
since $d_{\widehat{G}}(\phi(u),\phi(v))=d_G(u,v)$ when $u$ or $v$ is in $W$, but this equality is not
necessarily true for arbitrary $u,v\in V(G)$. 
We illustrate these concepts in the following example.
\begin{example}
    
     The set $S=\{ (0,2),(1,1),(2,0), (1,2), (2,1), (2,2)\}\subset \mathbb{Z}^2$ is realizable, and Figure~\ref{fig:canonical_a} shows the canonical realization of $S$, meanwhile a non-equivalent realization is shown in Figure~\ref{fig:canonical_b}. Both are depicted as subgraphs of $P_4\boxtimes P_4$ (see~\cite{MoraPuertas2023}). In both cases, the coordinates of every vertex are the distances to the vertices of $W=\{(0,2),(2,0)\}$. Clearly the second realization can be embedded into the canonical one, but not isometrically   
      since the distance between the vertices $(1,2)$ and $(2,2)$ is equal to 1 in the first case and equal to 3 in the second one.

\begin{figure}[ht]
     \centering
     \begin{subfigure}[b]{.25\textwidth}
     \centering
         \includegraphics[width=\textwidth]{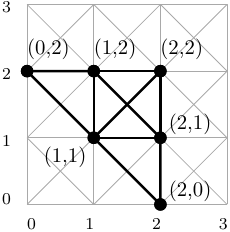}
         \caption{Canonical}
         \label{fig:canonical_a}
     \end{subfigure}
     \hspace{2cm}
     \begin{subfigure}[b]{.25\textwidth}
     \centering
         \includegraphics[width=\textwidth]{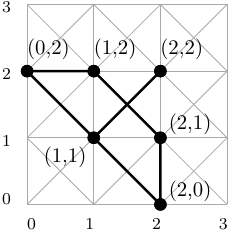}
         \caption{Non-canonical}
         \label{fig:canonical_b}
     \end{subfigure}
        \caption{Non-canonical realization as a spanning subgraph of the canonical one.}
        \label{fig:canonical}
\end{figure}

\end{example}

\section{Minimal realizations}\label{sec:minimal}

As we have stated in Section~\ref{Sec:intro}, all the graphs realizing $S$ can be considered as spanning subgraphs of the canonical realization $(\widehat{G},\widehat{W})$ of $S$. Therefore, the set of edges of the canonical realization is maximum with respect to the inclusion of the edge set.
The goal of this section is the study of minimal realizations of a subset $S\subset \mathbb{Z}^n$, that is, realizations such that the removal of some edges gives rise to a graph that does not realize $S$. 
Among these, we are also interested in determining those with the minimum number of edges.
Also, this leads to determine the sets that are uniquely realizable, that is, subsets such that all realizations are equivalent.

We begin by studying when a realization of $S$ remains so after the addition of an edge to the graph.

\begin{proposition}\label{pro:add_edge} Let $S\subset \mathbb{Z}^n$ be a realizable set.
If $(G,W)$ is a realization of $S$ and $xy\notin E(G)$, then $(G+xy,W)$ is a realization of $S$ if and only if $\max\limits_{i\in [n]}|x_i-y_i|=1$.
\end{proposition}

\begin{proof}
We consider $G$ as a spanning subgraph of the canonical realization $\widehat{G}$, that is $V(G)=S$, $E(G)\subset E(\widehat{G})$, and we denote 
$W=\{\omega^1,\dots \omega^n\}$. 

Firstly, given $xy\notin E(G)$, if $(G+xy,W)$ is a realization of $S$ then, $G+xy$ is also a spanning subgraph of $\widehat{G}$ and therefore $xy\in E(\widehat{G})$, that means $\max_{i\in [n]}|x_i-y_i|=1$. 

Conversely, assume that $\max_{i\in [n]}|x_i-y_i|=1$, so $xy\in E(\widehat{G})$, and let $ \omega^i\in W$ and $ z\in V(G+xy)$. Then, $d_{G+xy}(z,\omega^i)=  d_{\widehat{G}}(z,\omega^i)=z_i $, because $xy\in E(\widehat{G})$ and $\widehat{G}$ realizes $S$. This means that $r_{G+xy}(z\vert W)=(z_1, \dots, z_n)$, for every $z\in S$ and thus $(G+xy, W)$ realizes $S$.     
\end{proof}

This result tells us that, in order to maintain the status of realization of $S$, only those edges belonging to the canonical realization can be added to the graph $G$. Therefore, the problem of identifying which edges can be added is completely solved.

Now, we analyse when the removal of an edge of a realization $(G,W)$ of a set $S$ 
gives rise to another realization of the same set. This happens when the distance between any vertex of $W$ and any vertex of $G$ does not increase after the removal of the edge.

\begin{proposition}\label{pro:remove_edge} 
Let $S\subset \mathbb{Z}^n$ be a realizable set.
If $(G,W)$ is a realization of $S$ and $xy\in E(G)$, then 
$r_{G-xy}(u|W)=r_G(u|W)=(u_1,\dots ,u_n)$ for every $u\in S$
if and only if for every $i\in [n]$ such that $x_i=y_i-1$ (resp. $y_i=x_i-1$) there exists $z\in N_G( y)\setminus \{x\}$ (resp. $z\in N_G(x)\setminus \{y\}$) with
$z_i=x_i$ (resp. $z_i=y_i$). 
\end{proposition}
\begin{proof}
Let $(G,W)$ be a realization of $S$ such that $xy\in E(G)$ and we denote $W=\{\omega^1,\dots \omega^n\}$.

Suppose first that $r_{G-xy}(u|W)=r_G(u|W)=(u_1,\dots ,u_n)$ for every $u\in S$.
Let $i\in [n]$ and suppose $x_i=y_i-1$. 
Since
$d_{G-xy}(\omega^i,y)=d_G(\omega^i,y)=y_i$, there is a shortest path $\omega^i,z^1,\dots, z^{r-1},y$ in $G-xy$, where $r=y_i$. Hence, 
$d_G(\omega^i,z^{r-1})=d_{G-xy}(\omega^i,z^{r-1})=r-1=y_i-1=x_i$. Thus, $z=z^{r-1}\in N_G(y)-\{x\}$ and $z_i=x_i$.
Similarly, if  $y_i=x_i-1$ we prove that there exists $z\in S$
such that $z\in N_G(x)-\{y\}$ and $z_i=y_i$.

Now suppose that for every $i\in [n]$ such that $x_i=y_i-1$ (resp. $y_i=x_i-1$) there exists $z\in N_G( y)\setminus \{x\}$ (resp. $z\in N_G(x)\setminus \{y\}$) such that $z_i=x_i$ (resp. $z_i=y_i$). 
We claim that  $d_{G-xy}(\omega^i, u)=d_G(\omega^i, u)=u_i$, for every $u\in V(G)=V(G-xy)=S$.
Indeed, if there is a shortest path from $\omega^i$ to $u$ in $G$ not passing through $xy$, then $d_{G-xy}(\omega^i,u)=d_G(\omega^i,u)=u_i$. Otherwise, 
there is a path of length $u_i$ that goes through $xy$. Assume, without loss of generality, that the path is of the form $\omega^i,\dots ,x,y,\dots ,u$, so that $x_i=y_i-1$. By hypothesis,
there exists $z\in N_G(y)$, $z\not= x$,  such that $z_i=x_i=y_i-1$.
Hence, there is a path of the form $\omega^i,\dots ,z,y,\dots ,u$ 
from $\omega^i$ to $z$ in $G-xy$ and, since $d_{G-xy}(\omega^i,z)=d_{G}(\omega^i,z)=z_i=x_i,$ we have $d_{G-xy}(\omega^i,u)=d_G(\omega^i, u)=u_i$.
\end{proof}

The previous proposition provides necessary and sufficient conditions for a realization not to be minimal, therefore it also allows us to identify those that are minimal. Moreover, it provides us an algorithm to successively remove edges from the canonical realization $\widehat{G}$ to get a minimal realization. 

The following example shows that minimal realizations are not necessarily equivalent, even if they have the same number of edges, and that, in general, minimal realizations are not minimum.

\begin{example}
   The graphs depicted in Figure~\ref{fig:minimal} are non equivalent realizations of the set 
   $S=\{ (0,2),(1,1),(1,2), (2,0), (2,1), (2,2)\}$. In Figure~\ref{fig:realization_tree_a} we show the canonical realization and the rest are non-equivalent minimal realizations. Realization in Figure~\ref{fig:realization_tree_b} has size 8 and both in Figures~\ref{fig:realization_tree_c} and~\ref{fig:realization_tree_d} have size 7. Finally, in Figure~\ref{fig:realization_tree_e} we show a realization with minimum size equal to 6.
   
\begin{figure}[ht]
     \centering
     \begin{subfigure}{.13\textwidth}
     \centering
         \includegraphics[width=\textwidth]{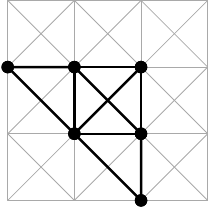}
         \caption{}
         \label{fig:realization_tree_a}
     \end{subfigure}
      \begin{subfigure}{.13\textwidth}
     \centering
         \includegraphics[width=\textwidth]{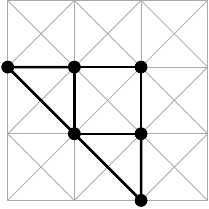}
         \caption{}
         \label{fig:realization_tree_b}
     \end{subfigure}
      \begin{subfigure}{.13\textwidth}
     \centering
         \includegraphics[width=\textwidth]{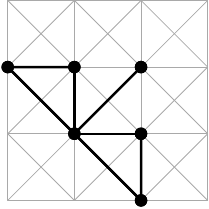}
         \caption{}
         \label{fig:realization_tree_c}
     \end{subfigure}
      \begin{subfigure}{.13\textwidth}
     \centering
         \includegraphics[width=\textwidth]{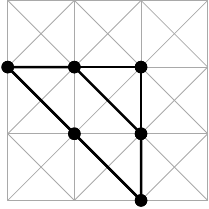}
         \caption{}
         \label{fig:realization_tree_d}
     \end{subfigure}
      \begin{subfigure}{.13\textwidth}
     \centering
         \includegraphics[width=\textwidth]{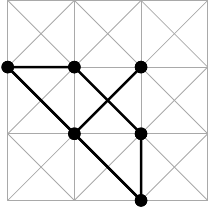}
         \caption{}
         \label{fig:realization_tree_e}
     \end{subfigure}
        \caption{The canonical realization and some non-equivalent minimal realizations of $S$.}
        \label{fig:minimal}
\end{figure}
\end{example}
Since the canonical realization always exists and it contains all the possible edges of every realization, all the realizations of a set $S$ are equivalent, that is, $S$ is uniquely realizable if and only if no proper spanning subgraph of the canonical one provides a realization of $S$. Proposition~\ref{pro:remove_edge} allows us to characterize these sets, as we show in the following theorem. We will need the following notation. For $x\in \mathbb{Z}^n$ and $S\subset \mathbb{Z}^n$, denote $D_S(x)=\{ y\in S: \max_{i\in [n]}  \vert x_i-y_i \vert =1\}$.

\begin{theorem}\label{thm:uniquely}
    If $S\subset \mathbb{Z}^n$ is  a realizable set, then all the realizations of $S$  are equivalent if and only if for every $x,y\in S$ with $\max_{i\in [n]} |x_i-y_i|=1$, there exists $i\in[n]$ such that either $x_i<y_i$ and 
$\{z\in D_S(y) : z_i=y_i-1\}=\{x\}$,
     or 
     $y_i<x_i$ and 
    $\{z\in D_S(x) : z_i=x_i-1\}=\{y\}$.
\end{theorem}
\begin{proof}
Propositions~\ref{pro:add_edge} and ~\ref{pro:remove_edge} imply that all the realizations of $S$ are equivalent to the canonical one if and only if 
$r_{G-e}(u|W)\not= r_G(u|W)$
for some edge from $e\in E(\widehat{G})$ and some vertex $u\in S$. Recall that for every $x,y\in S$, $xy$ is an edge of $\widehat{G}$ if and only if  $\max_{i\in [n]} |x_i-y_i|=1$. 
Hence, all the realizations of $S$ are equivalent if and only if for every $x,y \in S$ such that $\max_{i\in [n]} |x_i-y_i|=1$, there exists $u\in S$ such that
$r_{\widehat{G}-{xy}}(u|W)\not= r_{\widehat{G}}(u|W)$.
By Proposition~\ref{pro:remove_edge}, this happens if and only if there exists $i\in [n]$ such that either $x_i=y_i-1$ and 
$x$ is the only element in $D_S(y)$ with its ith coordinate equal to $y_i-1$,
or $y_i=x_i-1$ and $y$ is the only element in $D_S(x)$ with its i-th coordinate equal to $x_i-1$, as desired.
\end{proof}

\begin{example}
 The set $S=\{(0,2), (1,1), (1,3), (2,0),  (2,4), (3,1), (3,3), $  $(4,2) \}$ is uniquely realizable, since  the conditions of Theorem~\ref{thm:uniquely} are satisfied. The graph that realizes $S$ is a cycle (see Figure~\ref{fig:unique}). 

    \begin{figure}[ht]
       \centering
\includegraphics[width=0.26\linewidth]{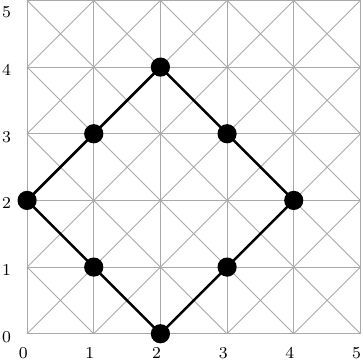}
\caption{The unique realization of the set $S=\{(0,2), (1,1), (1,3), (2,0),  $ $(2,4), (3,1), (3,3), (4,2) \}$.}
\label{fig:unique}
\end{figure}
\end{example}

The following example shows that, similarly to minimal realizations, minimum realizations are not necessarily equivalent.

\begin{example}
  In Figure~\ref{fig:minimum_canonical_2d} we show the canonical realization of the set $S=\{ (0,2), (1,1), (1,3), (2,0), (2,3), (2,4), (3,1), (3,2),(3,3), (4,2)\}$. We have represented with continuous lines the edges that belong to every realization of $S$, because they are in the unique shortest path between a vertex of the graph and a vertex of $W=\{ (0,2), (2,0)\}$, and with dashed lines the rest of the edges. 
  With this in mind, it is not difficult to check that graphs in Figures~\ref{fig:minimum_a_2d} and ~\ref{fig:minimum_b_2d} are two spanning subgraphs of the canonical realization, with the minimum number of edges and that also realize $S$. Moreover, it is clear that they are not isomorphic graphs.

\end{example}
\par\bigskip

\begin{figure}[ht]
     \centering
     \begin{subfigure}[t]{.28\textwidth}
     \centering
         \includegraphics[width=\textwidth]{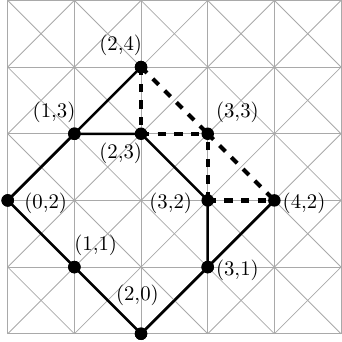}
         \caption{Canonical}
         \label{fig:minimum_canonical_2d}
     \end{subfigure}
     \quad
     \begin{subfigure}[t]{.28\textwidth}
     \centering
         \includegraphics[width=\textwidth]{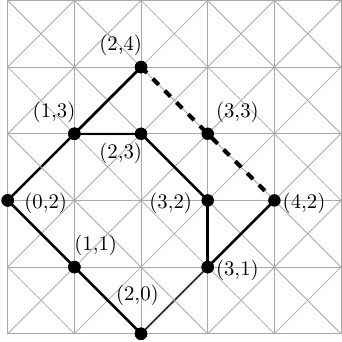}
         \caption{Minimum}
         \label{fig:minimum_a_2d}
     \end{subfigure}
     \quad
     \begin{subfigure}[t]{.28\textwidth}
     \centering
         \includegraphics[width=\textwidth]{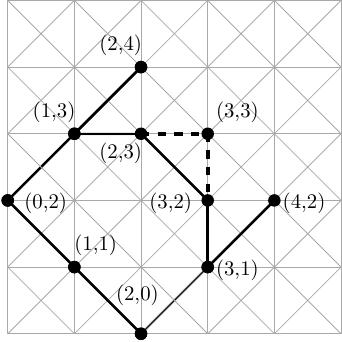}
         \caption{Minimum}
         \label{fig:minimum_b_2d}
     \end{subfigure}
        \caption{Three non-equivalent realizations of $S$.}
        \label{fig:minimum}
\end{figure}

\section{Computational complexity}\label{sec:complexity}
We now focus on the computational complexity of some decision problems regarding to the metric realization of a set of vectors. We begin with the basic realization problem.

\begin{definition}
    We define the Metric Realization Problem $\mathit{METREL}$ as the decision problem of, given a set $S\subset \mathbb{Z}^r$, determine if the set is realizable.
\end{definition}

By using Theorem \ref{theo:realizable}, it is clear that $\mathit{METREL}$ is in $\mathbf{P}$, because the three conditions to be realizable can easily be checked in polynomial time. Moreover, the canonical realization can also be obtained in polynomial time.

By including the idea of minimality of the realization, we can also consider the following decision problem. 

\begin{definition}
    The Bounded Metric Realization Problem $\mathit{BMETREL}$ is the decision problem of, given  a realizable set $S\subset \mathbb{Z}^r$ and a constant $k$, determine if the set is realizable by a graph with at most $k$ edges.
\end{definition}

We prove that the bounded problem $\mathit{BMETREL}$ is $\mathbf{NP-Complete}$, by a reduction of $\mathit{3SAT}$.
\begin{theorem}\label{th:bmetrel}
    $\mathit{BMETREL}$ is $\mathbf{NP-Complete}$.
\end{theorem}
\begin{proof}
    To prove $\mathit{BMETREL} \in \mathbf{NP}$ we just need to prove that there exists a witness for the problem that can be verified in polynomial time. Given a realizable set $S \subset \mathbb{Z}^n$ and a bound on the number of edges $k$, a witness for the existence of a valid realization is merely the realization itself. It can be easily verified in polynomial time to be a realization and that it does not exceed the number of edges. To prove $\mathit{BMETREL}$ is $\mathbf{NP-Hard}$ we shall reduce $\mathit{3SAT}$ to $\mathit{BMETREL}$ in polynomial time.
    
    Firstly, let $f$ be a $\mathit{3SAT}$ formula with $n$ variables $x_1, x_2, \dots, x_n$ and $m$ clauses $C_1, C_2, \dots, C_m$, each being a logical disjunction of $1$ to $3$ literals, with a literal being a variable $x_i$ or its negation $\overline{x_i}$.

    We shall assume, without loss of generality, that every variable appears with at least one other variable in some clause. If a variable appears in isolation we may get rid of it by unit propagating the variable into the formula. 
    
    To reduce from $\mathit{3SAT}$ to $\mathit{BMETREL}$ we shall produce a $\mathit{BMETREL}$ instance in polynomial time with equal truth value to the original $\mathit{3SAT}$ instance. To do so, we shall construct a realizable set $S \subset \mathbb{Z}^{r}$ such that $|S| = 3n + m + 2$ and $r = n + m + 1$. For ease of discussion, we shall adopt the following naming convention:
    \begin{itemize}
        \item There are $n$ elements in $S$  named $x_1, \dots, x_n$, and another $n$ elements named $\overline{x_1}, \dots, \overline{x_n}$.
         \item There are $n$ elements named $r_1, \dots, r_n$.
        \item There are $m$ elements named $c_1, \dots, c_m$.
        \item One element of $S$ is named $s$, and another one 
        named $t$.
    \end{itemize}
So $S=\{x_1, \dots, x_n, \overline{x_1}, \dots, \overline{x_n}, r_1,\dots, r_n, c_1, \dots, c_m, s,t\}$ and $\vert S\vert =3n+m+2$. Each $u=(u_1,\dots , u_{n+m+1})\in S\subset \mathbb{Z}^{n+m+1}$ is a vector with $n+m+1$ coordinates and in order to describe them, we denote $u_i=\delta(r_i,u)$ for $i\in[n]$, $u_{n+j}=\delta(c_j,u)$, for $j\in[m]$ and $u_{n+m+1}=\delta(s,u)$.
We now describe the coordinates of every element of $S$.

\begin{itemize}
     \item Coordinates $u_i$, $i\in[n]$:
     
        $
        \begin{array}{l}
             \delta(r_i, x_i) = \delta(r_i, \overline{x_i}) = 1\\
             \delta(r_i, x_p) = \delta(r_i, \overline{x_p}) = 3 \text{ for } i \neq p\\
             \delta(r_i, r_i) = 0 \\
              \delta(r_i, r_p) = 4 \text{ for } i \neq p\\
             \delta(r_i, c_j) =  \left\{\begin{matrix}
                2 & \text{if }x_i \in C_j \text{ or } \overline{x_i} \in C_j\\
                4 & \text{otherwise} \\
            \end{matrix}\right.\\
             \delta(r_i, s) = \delta(r_i, t) = 2.
        \end{array}
        $

 \item Coordinates $u_{n+j}$ with $j\in[m]$:
     
  $
        \begin{array}{l}
              \delta(c_j, l \in \{x_i, \overline {x_i}\}) =  \left\{\begin{matrix}
                1 & \text{if }l \in C_j \\
                3 & \text{if }l \notin C_j \\
            \end{matrix}\right.\\
             \delta(c_j, r_i) =  \left\{\begin{matrix}
                2 & \text{if }x_i \in C_j \text{ or } \overline{x_i} \in C_j\\
                4 & \text{otherwise} \\
            \end{matrix}\right.\\
        \end{array}\\
        \begin{array}{l}
             \delta(c_j, c_j) = 0\\
               \delta(c_j, c_p) =  \left\{\begin{matrix}
                2 & \text{if }j \neq p \text{ and }C_j \cap C_p \neq \emptyset \\
                4 & \text{if }j \neq p \text{ and }C_j \cap C_p = \emptyset \\
            \end{matrix}\right.\\
             \delta(c_j, s) = 
             \delta(c_j, t) = 2.\\      
         \end{array}
        $
        
\item Coordinate $u_{n+m+1}$: 
        
         $\begin{array}{l}
         \delta(s, x_i) = \delta(s, \overline{x_i}) = 1\\
         \delta(s, s) = 0\\
         \delta(s, u) = 2 \text{\ otherwise}.
         \end{array}
         $
    \end{itemize} 
    
By construction, the elements of $S$ that have a zero coordinate are the following: each $r_i$, with the $i-th$ coordinate equal to zero; each $c_j$, with the $n+j-th$ coordinate equal to zero; $s$, with the $n+m+1$ coordinate equal to zero.

We now prove that $S$ is realizable. To this end, we define the bipartite graph $G_0$ with vertex set $V(G_0)=S$, stable sets  $V_1=\{x_1,\dots x_n,\overline{x_1}, \dots \overline{x_n}\}$ and $V_2=\{r_1,\dots, r_n, c_1, \dots, c_m, s,t \}$, and edge set 
\begin{align*}
E(G_0)=
 & \{r_il\colon l\in \{x_i,\overline{x_i}\}, i\in [n]\} \cup \\
 & \{ c_jl\colon l\in \{x_i,\overline{x_i}\}, i\in [n] \text{ and } l\in C_j, j\in [m]\}\cup\\
 &\{sl\colon l\in \{x_i,\overline{x_i}\}, i\in [n]\}\cup \\ 
& \{ tl\colon l\in \{x_i,\overline{x_i}\}, i\in [n]\}.
\end{align*}

\begin{claim}
    The pair $(G_0, W)$ realizes $S$, where $W=\{r_1,\dots, r_n, c_1, \dots, c_m, s \}$.
\end{claim}
We will check that $S=\{ r(u\vert W) \colon u\in V(G_0)\}$. 

\begin{itemize}
\item $u=x_i, \overline{x_i}$:

Coordinates $u_i$, $i\in [n]$:\\
$u_i=\delta(r_i,u)=1=d_{G_0}(r_i,u)$, because they are neighbours in $G_0$; \\
if $p\neq i$ then, $u_p=\delta(r_p,u)=3=d_{G_0}(r_p,u)$ because they belong to different stable sets, so the distance is an odd number, they are not neighbours and the path $r_p-x_p-s-u$ has length $3$.   

Coordinates $u_{n+j}$, $j\in [m]$:\\
if $u\in C_j$, then $u_{n+j}=\delta(c_j,u)=1=d_{G_0}(c_j,u)$;\\
if $u\notin C_j$, then $u_{n+j}=\delta(c_j,u)=3=d_{G_0}(c_j,u)$ because they belong to different stable sets, so the distance is an odd number, they are not neighbours and the path $c_j-l-s-u$, with $l\in C_j$, has length $3$.   

Coordinate $u_{n+m+1}$:\\
$u_{n+m+1}=\delta(s,u)=1=d_{G_0}(s,u)$.

\item $u=r_i$:\\
Coordinates $u_i$, $i\in [n]$:\\
$u_i=\delta(r_i,r_i)=0=d_{G_0}(r_i,r_i)$; \\
if $p\neq i$ then, $u_p=\delta(r_p,r_i)=4=d_{G_0}(r_p,r_i)$ because they belong to the same stable set, so the distance is an even number, they have no common neighbour and the path $r_p-x_p-s-x_i-r_i$ has length $4$.

Coordinates $u_{n+j}$, $j\in [m]$:\\
if $x_i\in C_j$ or $\overline{x_i}\in C_j$, then $u_{n+j}=\delta(c_j,r_i)=2=d_{G_0}(c_j,r_i)$ because they are not neighbours and they have a common neighbour ($x_i$ or $\overline{x_i}$);\\
if $x_i\notin C_j$ and $\overline{x_i}\notin C_j$, then $u_{n+j}=\delta(c_j,r_i)=4=d_{G_0}(c_j,r_i)$ because they belong to the same stable set, so the distance is even, they have no common neighbour and the path $c_j-l-s-x_i-r_i$, with $l\in C_j$, has length $4$.   

Coordinate $u_{n+m+1}$:\\
$u_{n+m+1}=\delta(s,r_i)=2=d_{G_0}(s,r_i)$ because they are not neighbours and they have common neighbours (both $x_i$ and $\overline{x_i}$).

\item $u=c_j$:\\
Coordinates $u_i$, $i\in [n]$:\\
if $x_i\in C_j$ or $\overline{x_i}\in C_j$, then $u_{i}=\delta(r_i,c_j)=2=d_{G_0}(r_i,c_j)$ because they are not neighbours and they have a common neighbour ($x_i$ or $\overline{x_i}$);\\
if $x_i\notin C_j$ and $\overline{x_i}\notin C_j$, then $u_{i}=\delta(r_i,c_j)=4=d_{G_0}(r_i,c_j)$ because they belong to the same stable set, so the distance is an even number, they have no common neighbours and the path $r_i-s-x_i-l-c_j$, with $l\in C_j$, has length $4$.   

Coordinates $u_{n+j}$, $j\in [m]$:\\
$u_{n+j}=\delta(c_j,c_j)=0=d_{G_0}(c_j,c_j)$;\\
if $p\neq j$ and $C_p\cap C_j\neq \emptyset$ then, $u_{n+p}=\delta(c_p,c_j)=2=d_{G_0}(c_p,c_j)$ because they belong to the same stable set, so the distance is an even number, and they have common neighbours (every literal in $C_p\cap C_j$);\\
if $p\neq j$ and $C_p\cap C_j= \emptyset$ then, $u_{n+p}=\delta(c_p,c_j)=4=d_{G_0}(c_p,c_j)$ because they belong to the same stable set, so the distance is an even number, they have no common neighbours and the path $c_p-a-s-b-c_j$, with $a\in C_p$ and $b\in C_j$, has length $4$.   

Coordinate $u_{n+m+1}$:\\
$u_{n+m+1}=\delta(s,c_j)=2=d_{G_0}(s,c_j)$ because they are not neighbours and they have common neighbours (every literal in $C_j$).

\item $u=s,t$:\\
Coordinates $u_i$, $i\in [n]$:\\
$u_i=\delta(r_i,u)=2=d_{G_0}(r_i,u)$ because they are not neighbours and they have common neighbours (both $x_i$ and $\overline{x_i}$);\\

Coordinates $u_{n+j}$, $j\in [m]$:\\
$u_{n+j}=\delta(c_j,u)=2=d_{G_0}(c_j,u)$ because they are not neighbours and they have common neighbours (every literal in $C_j$);\\

Coordinate $u_{n+m+1}$:\\
$s_{n+m+1}=\delta(s,s)=0=d_{G_0}(s,s)$;\\
$t_{n+m+1}=\delta(s,t)=2=d_{G_0}(s,t)$ because they are not neighbours and they have common neighbours (every $x_i$ and every $\overline{x_i}$).
\end{itemize}

 Finally, to get an instance of $\mathit{BMETREL}$ decision problem, we set the integer $k$ as:    
 
    $$
    k = 5n + \sum_{j = 1}^m |C_j|.
    $$

  Note that $\vert E(G_0) \vert =6n + \sum_{j = 1}^m |C_j|>k$, so $G_0$ is not a solution of the instance of the $\mathit{BMETREL}$ decision problem.

Assume now that this instance of $\mathit{BMETREL}$ is true, then $S$ is realizable by graph $G$ with $V(G)=S$, $\vert E(G)\vert \leq k$ and, by our election of the order of the coordinates, the resolving set is $W=\{r_1,\dots, r_n, c_1, \dots, c_m, s \}$. 

So, for each $u\in V(G)$, we know that $d_G(r_i,u)=\delta (r_i,u)$, $d_G(c_j,u)=\delta (c_j,u)$ and $d_G(s,u)=\delta (s,u)$. This means that the degree of vertex $s$ is $2n$, because it is at distance $1$ from each $x_i$ and each $\overline{x_i}$, the degree of each vertex $r_i$ is $2$, because it is at distance $1$ from $x_i$ and $\overline{x_i}$ and the degree of each vertex $c_j$ is $\vert C_j\vert$, because it is at distance $1$ from each element of $C_j$. So, we have $4n+\sum_{i=j}^m |C_j|$ edges in $G$, which have one vertex in $W$ and the other one in $\{x_1,\dots x_n, \overline{x_1}, \dots \overline{x_n}\}$.
Now, for each $i\in [n]$ we know that $d_G(r_i,t)=2$, so $t$ is a neighbour of $x_i$ or $\overline{x_i}$, which are the unique neighbours of $r_i$. By using that $\vert E(G)\vert \leq k=5n + \sum_{j = 1}^m |C_j|$, we obtain that $t$ is a neighbour of exactly one vertex in $\{x_i, \overline{x_i}\}$, for each $i\in [n]$. These are all the edges of $G$.

Moreover, for each $j\in [m]$, $d_G(c_j,t)=2$ and using that the neighbours of $c_j$ are the literals that form the clause $C_j$, this implies that at least one of the literals that forms the constraint $C_j$ is a neighbour of $t$. Taking the $n$ literals that are neighbours of $t$, we get precisely an assignment that satisfies the original $3SAT$ instance.

Conversely, we now assume that the $3SAT$ formula $f$ is satisfiable, and we have to prove that the instance of $\mathit{BMETREL}$ is realizable by a graph with at most $k$ edges. We consider the graph $G$ with vertex set $V(G)=S$ and edge set 
\begin{align*}
E(G)=
 & \{r_il\colon l\in \{x_i,\overline{x_i}\}, i\in [n]\} \cup \\
 & \{ c_jl\colon l\in \{x_i,\overline{x_i}\}, i\in [n] \text{ and } l\in C_j, j\in [m]\}\cup\\
 &\{sl\colon l\in \{x_i,\overline{x_i}\}, i\in [n]\}\cup \\ 
& \{ tl\colon l\in \{x_i,\overline{x_i}\}, i\in [n] \text{ and } l \text{ is true}\}.
\end{align*}
Clearly $\vert E(G)\vert=5n + \sum_{j = 1}^m |C_j|=k$. We take $W=\{r_1,\dots, r_n, c_1, \dots, c_m, s \}$ and we have to check that $(G,W)$ realizes $S$.

Note that the edge set of $G$ is the same as the edge set of $G_0$, except for the edges containing vertex $t$. Moreover, the reasoning above about the coordinates of vertices $u\in V(G_0)=V(G)=S, u\neq t$, does not include any edge containing $t$. This means that in the graph $G$, every vertex $u\neq t$ satisfies $u=(u\vert W)$. Finally, we check this condition also for $t$:\\
Coordinates $t_i$, $i\in [n]$:\\
$t_i=\delta(r_i,t)=2=d_{G}(r_i,t)$ because they are not neighbours in $G$ and they have a common neighbour ($x_i$ or $\overline{x_i}$, the true one);\\
Coordinates $t_{n+j}$, $j\in [m]$:\\
$t_{n+j}=\delta(c_j,t)=2=d_{G}(c_j,t)$ because they are not neighbours in $G$ and they have a common neighbour (each true literal in $C_j$);\\
Coordinate $t_{n+m+1}$:\\
$t_{n+m+1}=\delta(s,t)=2=d_{G}(s,t)$ because they are not neighbours $G$ and they have common neighbours ($x_i$ or $\overline{x_i}$, the true one, for every $i\in [n]$).

    This reduction produces a set $S$ of size $3n + m + 2$, with each element being a vector with $m + n + 1$ coordinates, making the output size $O(n^2 + m^2 + nm)$, which is polynomial with respect to the size of the formula $f$, and the whole reduction can easily be run in polynomial time.
\end{proof}

\section{Trees}\label{sec:trees}
In the context of the study of the minimal realizations of $S\subset\mathbb{Z}^n$, it is natural to ask about realizations that are trees. So, in this section we focus on characterizing sets that can be realized by a tree. This problem also appears in~\cite{Feit2018}.

We denote
$S_1=\{x\in S: (x_1-1,\dots,x_n-1)\in S \}$ and $S_0=S\setminus S_1$.
If $x\in \mathbb{Z}^n$ and $k\ge 1$, we denote $x-k=(x_1-k,\dots,x_n-k)$. With this terminology, $S_1=\{x\in S: x-1\in S\}$. In the following lemmas we show the role that the subset $S_0$ plays in the realization of a set $S$ by a tree. 

\begin{lemma}\label{lem:V0W}
Let $S\subset \mathbb{Z}^n$ and suppose that there is a realization $(T,W)$ of $S$ such that $T$ is  a tree. Then, $S_0$ is the set of vertices of $T$ belonging to some path with both end vertices in $W$.
\end{lemma}
\begin{proof}
Consider first a vertex $u$ belonging to some path in $T$ with end vertices in $W$. We claim $u\in 
S_0$. 

Clearly, if $u\in W$ we are done, so assume that $u$ is an interior vertex of the only path in $T$ from  $\omega^i$ to $ \omega^j$, then $d(\omega^i,\omega^j)=d(\omega ^i,u)+d(\omega^j, u)$. Suppose, on the contrary, that $u\in S_1$, that is $u-1\in S$. By using that $(T,W)$ realizes $S$, there exists $v\in V(T)=S$ such that $v=u-1$. Then, $d(\omega^i,\omega^j)\leq d(\omega^i,v)+d(v,\omega^j)=v_i+v_j=u_i-1+u_j-1=d(\omega ^i,u)+d(\omega^j, u)-2=d(\omega^i,\omega^j)-2$, a contradiction.

Suppose now that  $u\in S_0$. If $u\in W$, then $u$ is trivially a path in $S_0$ with end vertices in $W$. 
If $u\in S_0\setminus W$, let $v^1,\dots ,v^d$ be the neighbors of $u$ in $T$. 
Consider the connected components $C_1,\dots ,C_d$ of $T-u$, where $v^i\in V(C_i)$. There are at least two components with vertices in $W$, because otherwise $W\subseteq V(C_{i_0})$ for some $i_0\in [d]$, implying that $v^{i_0}=u-1$ and, consequently $u\in S_1=S\setminus S_0$, a contradiction. 
Hence, $u$ belongs to a path with end vertices in $W$, concretely, the vertices of $W$ belonging to different components of $T-u$.
\end{proof}

\begin{lemma}\label{lem:V0tree}
Let $S\subset \mathbb{Z}^n$ and suppose that there is a realization $(T,W)$ of $S$ such that $T$ is  a tree. Then, the set $S_0$ induces a tree in $T$ with all leaves belonging to $W$. 
\end{lemma}
\begin{proof}
On the one hand, Lemma~\ref{lem:V0W} implies that the set $S_0$ induces a connected graph in $T$  and, hence, a tree.
On the other hand, if $u\in S_0\setminus W$, then Lemma~\ref{lem:V0W} implies that $u$ has at least two neighbours  in $S_0$. Therefore, $u$ is not a leaf in the tree induced by $S_0$.
\end{proof}

We can now characterize the sets that are realizable by a tree.

\begin{theorem}\label{thm:charact}
If $S\subset \mathbb{Z}^n$ is a realizable set, then there exists a realization $(T,W)$ of $S$ such that $T$ is a tree if and only if
the following conditions hold
\begin{enumerate}[i)]
\item for every $x,y\in S$, if $\max_{i\in [n]}|x_i-y_i|=1$, then $y_j\not= x_j$ for every $j\in [n]$; 
\item for every $x\in S_0$ and every $j\in [n]$ such that $x_j>0$, there exists exactly one element $y\in S_0$
such that $\max_{i\in [n]}|x_i-y_i|=1$ and $y_j=x_j-1$.
\end{enumerate}  
\end{theorem}
\begin{proof} 
$\Rightarrow )$ Suppose first that $(T,W)$ is a realization of $S$, where $T$ is a tree.

Let $x,y\in S$ be such that $\max_{i\in [n]} |x_i-y_i|=1$, then $|x_{i_0}-y_{i_0}|=1$ for some $i_0\in [n]$. 
Then, $d(x,\omega^{i_0})=x_{i_0}$ and $d(y,\omega^{i_0})=y_{i_0}$. Since $x_{i_0}$ and $y_{i_0}$ have different parity, and using that the tree $T$ is a bipartite graph, $x$ and $y$ belong to different stable sets of $V(T)=S$. Hence, 
$d(\omega^j,x)=x_j$ and $d(\omega^j,y)=y_j$ have different parity for every $j\in [n]$.
Consequently, $x_j\not= y_j$ for every $j\in [n]$. Thus, 
$i)$ holds.

Now suppose that $x\in S_0$ and  $x_j>0$, for some $j\in [n]$. By  Lemma~\ref{lem:V0tree},  $S_0$ induces a tree $T_0$ in $T$ with all leaves belonging to $W$. Hence,  $y$ the neingbor of $x$ in the path in $T$ between $x$ and $\omega^j$ is the unique vertex in $S_0$ such that $d(y, \omega^j)=d(x,\omega^j)-1$. 
Thus, $y$ is the only element of $S_0$ such that $\max_{i\in [n]}|x_i-y_i|=1$ and $y_j=x_j-1$. Therefore, ii) holds.

$\Leftarrow )$ Suppose now that $S$ is a realizable set and conditions i) and ii) hold. Let $(\widehat{G},\widehat{W})$ be the canonical realization of $S$. We claim that $\widehat{G}[S_0]$ is a tree  such that $d_{\widehat{G}[S_0]}(x,y)=d_{\widehat{G}}(x,y)$, for every $x,y\in S_0$.
Indeed, $\widehat{G}[S_0]$ is connected, because by condition ii) there is a path in $\widehat{G}[S_0]$ from every vertex $x\in S_0$ to any vertex $\omega^i\in \widehat{W}$.
Besides, 
$\widehat{G}[S_0]$ has no cycles. Suppose on the contrary that $\widehat{G}[S_0]$ has a cycle $C$. Take $x\in V(C)$ with a coordinate that maximizes the coordinates of the vertices of $C$, that is, there exists $i_0\in [n]$, such that $x_{i_0}\ge z_j$ for every $z\in V(C)$ and every $j\in [n]$.
Let $y^1$ and $y^2$ be the neighbours of $x$ in $C$. Then, $x_{i_0}\ge y_{i_0}^1$ and $ x_{i_0}\ge y_{i_0}^2$, by the choice of $x$. Hence, by Condition i), 
$y_{i_0}^1=y_{i_0}^2= x_{i_0}-1$, but this is a contradiction, since Condition ii) holds. Therefore, $\widehat{G}[S_0]$ is a tree.

Condition ii) implies that from every $x\in S_0$ there is a shortest path in $\widehat{G}$ to every vertex in $\widehat{W}$ containing only vertices of $S_0$.
So, $d_{\widehat{G}[S_0]}(x,y)=d_{\widehat{G}}(x,y)$, if $x,y\in S_0$.

Now consider the graph $T$ with vertex set $V(T)=S$ and edge set 

$$E(T)=E(\widehat{G}[S_0])\cup \{yy': y\in S_1, y'=y-1\}.$$ 

We claim that $(T, \widehat{W})$ is a tree realizing $S$. 
Indeed, the set of edges of $T$ is well defined, since $y-1\in S$ when $y\in S_1$. 
Besides, $T[S_0]=\widehat{G}[S_0]$ is connected by Lemma~\ref{lem:V0tree} and, for every $y\in S_1$, there exists $k\ge 1$ such that  $y-1,\dots ,y-(k-1)\in S_1$ and $y-k\in S_0$. Hence,
$y,y-1,\dots ,y-k$ is a path in $T$ from $y$ to a vertex in $S_0$, implying that $T$ is connected. 
The size of $T$ is

$$|E(T)|=|E(\widehat{G}[S_0])|+|S_1|=|S_0|-1+|S_1|=|S|-1=|V(T)|-1.$$
Therefore, $T$ is a tree.

Let us prove now that $(T,\widehat{W})$ realizes $S$. 
Since $T[S_0]= \widehat{G}[S_0]$ is a tree,  then for every if $x,y\in S_0$

$$d_T(x,y)=d_{\widehat{G}[S_0]}(x,y)=d_{\widehat{G}}(x,y).$$ 
Hence, 
$r_T(x|\widehat{W})=r_{\widehat{G}}(x|\widehat{W})=x$, if $x\in S_0$.
If $x\in S_1$, then let $k\ge 1$ be the only natural number such that $x-k\in S_0$. 
By definition of $T$, $d_T(x,\omega^i)=d_T(x-k,\omega^i)+k=
d_{\widehat{G}}(x-k,\omega^i)+k=
(x_i-k)+k=x_i$.
Hence, $r_T(x|\widehat{W})=x$, if $x\in S_1$.
Consequently, $(T,\widehat{W})$ is a realization of $S$.
\end{proof}

The former theorem is similar to Lemma 3.1 of \cite{Feit2018}, however, we relax condition ii) of such lemma by showing that it is sufficient to consider it only for the vertices in $S_0$, instead of for all vertices of the graph.

Theorem~\ref{thm:charact} provides an explicit construction of a realization $(T,\widehat{W})$ of $S\subset \mathbb{Z}^n$, under the appropriate conditions, where $T$ is a tree. Recall that this tree satisfies $V(T)=S$ and $E(T)=E(\widehat{G}[S_0])\cup \{yy': y\in S_1, y'=y-1\}$. In the following result we show that this is the unique option to realize $S$ as a tree.

\begin{proposition}\label{prop:tree}
Let $S\subset \mathbb{Z}^n$ be a set realizable by a tree. Then, all the realizations of $S$ by trees are equivalent.
\end{proposition}
\begin{proof}
By Theorem~\ref{thm:charact}, the subtree $T$ of the canonical realization $\widehat{G}$ with $V(T)=S$ and $E(T)=E(\widehat{G}[S_0])\cup \{yy': y\in S_1, y'=y-1\})$, satisfies that $(T,\widehat{W})$ realizes $S$, where $\widehat{W}=\{x\in S\colon x_i=0, \text{ for some }i\in [n]\}$.

Let $(T',\widehat{W})$ be another realization of $S$, satisfying that $T'$ is a tree. In order to prove that both $T$ and $T'$ are equivalent realizations of $S$, we just need to show that $E(T)=E(T')$. 

Let $xy\in E(T')$ be an edge of $T'$, then $xy\in E(\widehat{G})$. If $x,y\in S_0$, then $xy\in E(\widehat{G}[S_0])\subseteq E(T)$. Otherwise, $\{x,y\}\cap S_1\neq \emptyset$ and, in such a case, we know that $\max_{i\in [n]}|x_i-y_i|=1$ and, by Theorem~\ref{thm:charact}, $x_j\neq y_j$ for every $j\in [n]$. Suppose without loss of generality that $x_1=y_1-1$, therefore $d_{T'}(x,\omega^1)=x_1=y_1-1=d_{T'}(y,\omega^1)-1$ and $x$ is in the path in $T'$ between $y$ and $\omega^1$. Similarly, if $x_j=y_j+1$, for some $j\geq 2$, then $y$ is in the path in $T'$ between $x$ and $\omega^j$. Then, there is a path in $T'$ between $\omega^1$ and $\omega^j$ containing both $x$ and $y$, which contradicts Lemma~\ref{lem:V0W}. This means that, if $x_1=y_1-1$ then $x=y-1$, $y\in S_1$ and $xy\in \{yy': y\in S_1, y'=y-1\}\subseteq E(T)$.

Therefore, we have obtained that $E(T')\subseteq E(T)$. Finally, $\vert S\vert-1=\vert V(T')\vert-1=\vert E(T')\vert \leq \vert E(T)\vert =\vert V(T')\vert-1 =\vert S\vert-1$ gives that $E(T')=E(T)$, as desired. 
\end{proof}

The above proposition essentially provides the same result as Theorem 3.2 of~\cite{Feit2018}, however our proof, which uses the specific construction shown in the Theorem~\ref{thm:uniquely}, is simpler and shorter.

Although all the realizations of a subset $S$ by a tree are equivalent, if $S$ is realizable by a tree, it does not imply that  $S$ is uniquely realizable, that is, realizations of $S$ by graphs that are not trees may exist, as can be seen in Example~\ref{ex:tree}.

\begin{example}\label{ex:tree}
    It is easy to check that the set $S=\{ (0,3), (1,2), (2,1), (2,3),$ $(3,0),(3,2)\}$ is realizable, that is, it satisfies conditions of Theorem~\ref{theo:realizable}, and moreover, is is realizable by a tree, because it also satisfies conditions of Theorem~\ref{thm:charact}.   Figure~\ref{fig:realization_tree} shows the tree realization of the set $S$, given by Proposition~\ref{prop:unique_realizable_tree}. However, this realization is not unique and a non-equivalent realization is shown in Figure~\ref{fig:realization_non_tree}.

\end{example}

\begin{figure}[ht]
     \centering
     \begin{subfigure}[t]{.26\textwidth}
     \centering
         \includegraphics[width=\textwidth]{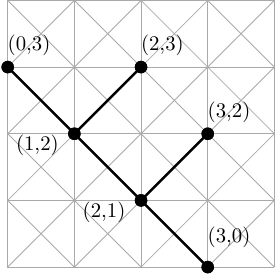}
         \caption{Tree realization}
         \label{fig:realization_tree}
     \end{subfigure}
     \hspace{2cm}
     \begin{subfigure}[t]{.26\textwidth}
     \centering
         \includegraphics[width=\textwidth]{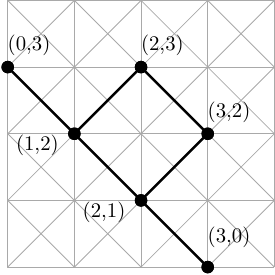}
         \caption{Non-tree realization}
         \label{fig:realization_non_tree}
     \end{subfigure}
        \caption{Example of two realizations, one of them is a tree.}
        \label{fig:realizations}
\end{figure}

Question 1 posed in~\cite{Feit2018} asks about the uniqueness of the realization of sets realizable by a tree. In the following result we provide a characterization of such sets, that is, the sets that are uniquely realizable and its realization is a tree.

\begin{proposition}\label{prop:unique_realizable_tree}
If $S\subset \mathbb{Z}^n$ is realizable by a tree then, it is uniquely realizable if and only if $\max_{i\in [n]} |x_i-y_i|>1$ for every pair of different vertices $x,y\in S_0^*$, where $S_0^*=\{x\in S_0: x+1\in S\}$.
\end{proposition}

\begin{proof}

$\Rightarrow )$
Assume first $S$ is uniquely realizable by a tree, then it is clear that the canonical realization $\widehat{G}$ is a tree. Suppose, on the contrary, that there exist different vertices $x,y\in S_0^*$ such that 
$\max_{i\in [n]} |x_i-y_j|=1$. Then, $x'=x+1$ and $y'=y+1$ belong to $S$ and $x'y'$ is an edge of $\widehat{G}$, since 

$$\max_{i\in [n]} |x'_i-y_i'|=\max_{i\in [n]} |(x_i+1)-(y_i+1)|=\max_{i\in [n]} |x_i-y_j|=1.$$ 
Hence, $xx'y'yx$ is a cycle in $\widehat{G}$, which contradicts that $\widehat{G}$ is a tree.

$\Leftarrow )$ For the converse, it is enough to check that the canonical realization $\widehat{G}$ is a tree. 
If $S$ is realizable by a tree, then $\widehat{G}$ has no cycles with all vertices in $S_0$,  since we have seen in the proof of Theorem~\ref{thm:charact} that $\widehat{G}[S_0]$ is a tree. Hence, the existence of a cycle in $\widehat{G}$ implies the existence of an edge $x'y'$ in $\widehat{G}$ with at least one of the vertices $x'$ or $y'$ belonging to $S_1$. 
Thus, we may assume without loss of generality that there exist integers $r$ and $s$ with $r\ge s\ge 0$ and $r\ge 1$, such that $x=x'-r\in S_0$  and $y=y'-s\in S_0$. 
If $d_{\widehat{G}}(x,y)=d$, then there exists $j_0\in [n]$ such that
$x_{j_0}-y_{j_0}=d$, because $\widehat{G}[S_0]$ is a tree with all leaves in $W$.
Hence, $x'_{j_0}-y'_{j_0}=
(x_{j_0}+r)-(y_{j_0}+s)=
x_{j_0}-y_{j_0}+(r-s)=d+(r-s)$.

Therefore, $\max_{i\in [n]}|x_i'-y_i'|\ge d+(r-s)>1$, since either $d\ge 2$, when $x',y'\in S_1$, or $d\ge 1$, $r\ge 1$ and $s=0$, when $x'\in S_1$ and $y'\in S_0$, which contradicts that $x'y'$ is an edge of $\widehat{G}$.
\end{proof}

In the following example we show a set $S$ which is uniquely realizable and its realization is a tree.

\begin{example}
    The set $S=\{ (0,4), (1,3), (2,2), (2,4), (3,1), (4,0), (4,2)\}$ satisfies conditions of Proposition~\ref{prop:unique_realizable_tree}, where $S_0^*=\{(1,3), (3,1) \}$ and, therefore, its unique realization, which is the canonical one, is a tree. We show it in Figure~\ref{fig:unique_realization_tree}. 
\end{example}

\begin{figure}[ht]
    \centering
         \includegraphics[width=0.28\textwidth]{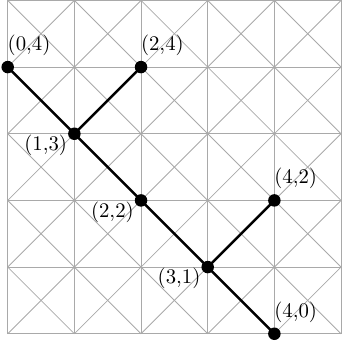}
         \caption{$S_0^*=\{(1,3), (3,1) \}$}
         \label{fig:unique_realization_tree}
\end{figure}

\bibliographystyle{abbrv}
\bibliography{bibfile}

\end{document}